\documentclass[12pt]{amsart}

\usepackage[text={425pt,650pt},centering]{geometry}
\usepackage{color}
\usepackage{esint,amssymb}
\usepackage{graphicx}
\usepackage{tikz}

\newtheorem{theorem}{Theorem}
\newtheorem{proposition}[theorem]{Proposition}
\newtheorem{lemma}[theorem]{Lemma}

\theoremstyle{definition}

\numberwithin{equation}{section}
\numberwithin{figure}{section}
\numberwithin{theorem}{section}

\newcommand{\R}{\mathbb{R}}

\newcommand{\E}{\mathbb{E}}
\renewcommand{\P}{\mathbb{P}}
\newcommand{\F}{\mathcal{F}}
\newcommand{\A}{\mathcal{A}}

\newcommand{\Rd}{\mathbb{R}^d}

\newcommand{\ep}{\varepsilon}

\DeclareMathOperator{\USC}{USC}
\DeclareMathOperator{\BUC}{BUC}

\renewcommand{\tilde}{\widetilde}

\begin{document}

\title[Homogenization of a nonconvex Hamilton-Jacobi equation]{Stochastic homogenization of a nonconvex Hamilton-Jacobi equation}

\begin{abstract}
We present a proof of qualitative stochastic homogenization for a nonconvex Hamilton-Jacobi equation. The new idea is to introduce a family of ``sub-equations" and to control solutions of the original equation by the maximal subsolutions of the latter, which have deterministic limits by  the subadditive ergodic theorem and maximality. 
\end{abstract}

\author[S. N. Armstrong]{Scott N. Armstrong}
\address{Ceremade (UMR CNRS 7534), Universit\'e Paris-Dauphine, Paris, France}
\email{armstrong@ceremade.dauphine.fr}

\author[H. V. Tran]{Hung V. Tran}
\address{Department of Mathematics\\
The University of Chicago\\ 5734 S. University Avenue Chicago, Illinois 60637, USA}
\email{hung@math.uchicago.edu}

\author[Y. Yu]{Yifeng Yu}
\address{Department of Mathematics\\
University of California at Irvine, California 92697, USA}
\email{yyu1@math.uci.edu}

\date{\today}

\maketitle

\section{Introduction}

\subsection{Motivation and overview}
We study the Hamilton-Jacobi equation
\begin{equation} \label{e.pde}
u^\ep_t + \left( \left| Du^\ep\right|^2 - 1 \right)^2 - V\!\left( \frac x\ep  \right) = 0 \quad \mbox{in} \ \Rd \times (0,\infty), \quad d \geq 1.
\end{equation}
The potential $V$ is assumed to be a bounded, stationary--ergodic random potential. We prove that, in the limit as the length scale $\ep > 0$ of the correlations tends to zero, the solution $u^\ep$ of~\eqref{e.pde}, subject to an appropriate initial condition, converges to the solution $u$ of the effective, deterministic equation
\begin{equation} \label{e.pdehom}
u_t + \overline H(Du) = 0 \quad \mbox{in} \ \Rd\times (0,\infty). 
\end{equation}
The effective Hamiltonian $\overline H$ is, in general, a non-radial, nonconvex function whose graph inherits the basic ``mexican hat" shape of that of the spatially independent Hamiltonian $p\mapsto (|p|^2-1)^2$. As we will show, it typically has two ``flat spots" (regions in which it is constant), with one a neighborhood of the origin and at which $\overline H$ attains a local maximum and the other a neighborhood of $\{ |p| = 1\}$ and at which $\overline H$ attains its global minimum. See Figure~\ref{fig.effham1}.

\smallskip

Qualitative stochastic homogenization results for convex Hamilton-Jacobi equations were first obtained independently by~Rezakhanlou and Tarver~\cite{RT} and~Souganidis~\cite{S} and subsequent qualitative results were obtained by Lions and Souganidis~\cite{LS1,LS2,LS3}, Kosygina, Rezakhanlou and Varadhan~\cite{KRV}, Kosygina and Varadhan~\cite{KV}, Schwab~\cite{Sch}, Armstrong and Souganidis~\cite{AS1,AS2} and Armstrong and Tran~\cite{AT1}. Quantitative homogenization results were proved in Armstrong, Cardaliaguet and Souganidis~\cite{ACS} (see also Matic and Nolen~\cite{MN}).

\smallskip

In contrast to the periodic setting, in which nonconvex Hamiltonians are not more difficult to handle than convex Hamiltonians (c.f.~\cite{E2,LPV}), extending the results of~\cite{RT,S} to the nonconvex case has remained, until now, completely open (except for the quite modest extension to \emph{level-set convex} Hamiltonians~\cite{AS2} and the forthcoming work~\cite{CST}, which considers a first-order motion with a sign-changing velocity). The issue of whether convexity is necessary for homogenization in the random setting is mentioned prominently as an open problem for example in~\cite{Kosy,LS2,LS3}. As far as we know, in this paper we present the first stochastic homogenization result for a genuinely non-convex Hamilton-Jacobi equation.

\smallskip

A new proof of qualitative homogenization for convex Hamilton-Jacobi equations in random environments was introduced in~\cite{AS2}, based on comparison arguments which demonstrate that \emph{maximal subsolutions} of the equation (also called solutions of the \emph{metric problem}) control solutions of the \emph{approximate cell problem}. This new argument is applicable to merely level-set convex Hamiltonians and lead to the quantitative results of~\cite{AS2}, among other developments. Several of the comparison arguments we make in the proofs of Lemmas~\ref{l.below.+1}--\ref{l.above-hilltop}, below, which are at the core of the argument for our main result, rely on some of the ideas introduced in~\cite{AS2}. The metric problem was also used to obtain dynamical information in Davini and Siconolfi~\cite{DS1,DS2} and as the basis of numerical schemes for computing~$\overline H$ in Oberman,~Takei and Vladimirsky~\cite{OTV} and Luo, Yu and Zhao~\cite{LYZ}.

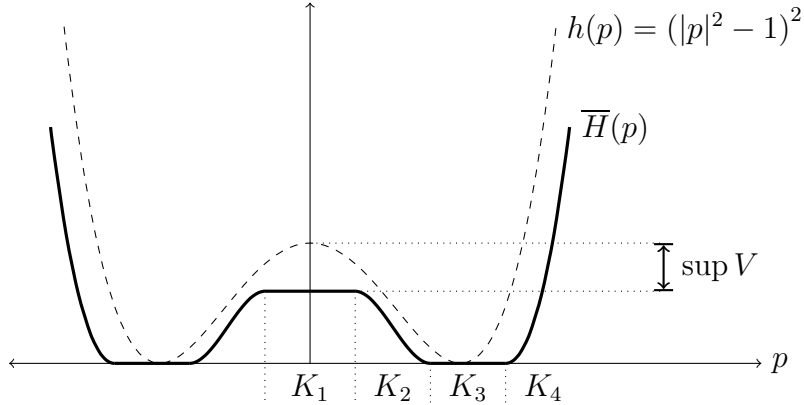
\begin{figure}
\label{fig.effham1}
\centering
\begin{tikzpicture}[yscale=1.6,xscale = 2]
    \draw[<->] (-2,0) -- (3,0) node[right] {$p$};
    \draw[->] (0,0) -- (0,3) node[above] {};
    \draw[black,dashed] plot [smooth] coordinates {(-1.635,2.8) (-1.6,2.434) (-1.52,1.717) (-1.44,1.153) (-1.36,0.722) (-1.28,0.408) (-1.2,0.194) (-1.12,0.065)(-1.04,0.007) (-1,0) (-0.96,0.006) (-0.88,0.051) (-0.8,0.13) (-0.72,0.232)(-0.64,0.349)(-0.56,0.471)(-0.48,0.592)(-0.4,0.706)(-0.32,0.806)(-0.24,0.888)(-0.16,0.949)(-0.08,0.987)(0,1)(0.08,0.987)(0.16,0.949)(0.24,0.888)(0.32,0.806)(0.4,0.706)(0.48,0.592)(0.56,0.471)(0.64,0.349)(0.72,0.232)(0.8,0.13)(0.88,0.051)(0.96,0.006)(1,0)(1.04,0.007)(1.12,0.065)(1.2,0.194)(1.28,0.408)(1.36,0.722)(1.44,1.153)(1.52,1.717)(1.6,2.434) (1.635,2.8)}node[right] {$h(p)=\left( |p|^2-1 \right)^2$};
    \draw[black, very thick] (-0.3,0.6)--(0.3,0.6);
        \draw[black, very thick] (-0.8,0)--(-1.3,0);
           \draw[black, very thick] (0.8,0)--(1.3,0);
    \draw[black, very thick] plot [smooth] coordinates { (0.3,0.6)
(0.31,0.599)(0.34,0.589)(0.38,0.559)(0.42,0.513)(0.48,0.423)(0.55,0.3)(0.6,0.211)(0.64,0.145)(0.68,0.087)(0.71,0.052)(0.74,0.024)(0.77,0.006)(0.79,0.001)(0.8,0)  } ;
    \draw[black, very thick] plot [smooth] coordinates { (-0.3,0.6)
(-0.31,0.599)(-0.34,0.589)(-0.38,0.559)(-0.42,0.513)(-0.48,0.423)(-0.55,0.3)(-0.6,0.211)(-0.64,0.145)(-0.68,0.087)(-0.71,0.052)(-0.74,0.024)(-0.77,0.006)(-0.79,0.001)(-0.8,0)  } ;
\draw[black,very thick] plot [smooth] coordinates { (1.3,0) (1.31,0.001)
(1.33,0.008) (1.36,0.033) (1.4,0.094) (1.46,0.247) (1.54,0.577) (1.62,1.063) (1.7,1.722) (1.725,1.965)} node[right]{$\overline H(p)$};
\draw[black,very thick] plot [smooth] coordinates { (-1.3,0) (-1.31,0.001)
(-1.33,0.008) (-1.36,0.033) (-1.4,0.094) (-1.46,0.247) (-1.54,0.577) (-1.62,1.063) (-1.7,1.722) (-1.725,1.965)};
\draw[dotted] (-0.3,0.6)--(-0.3,-0.35);
\draw[dotted] (0.3,0.6)--(0.3,-0.35);
\draw[dotted] (0.8,0)--(0.8,-0.35);
\draw[dotted] (1.3,0)--(1.3,-0.35);
\draw (0,0) node[below] 
      {$K_1$};
\draw (0.55,0) node[below] 
      {$K_2$};
\draw (1.05,0) node[below] 
      {$K_3$};
 \draw (1.55,0) node[below] 
      {$K_4$};
 \draw[dotted] (0,1) -- (2.35,1);     
  \draw[dotted] (0,0.6) -- (2.35,0.6);     
  \draw[thick,|<->|] (2.35,1)--(2.35,0.6);
  \draw (2.4,0.8) node[right] {$\sup V$};
    \end{tikzpicture}
    \caption{A cross section of the graph of $\overline H$, illustrated in the case $\inf V = 0$ and $\sup V = \frac25$. The difference $h(0) - (\overline H(0) - \inf V)$ is precisely $\max\{ 1, \sup V\}$. The regions $K_i$ are defined below in~\eqref{e.Ki}. While $\overline H$ is even, it is not radial, in general, unless for example the law of $V$ is invariant under rotations.}
\end{figure}

\smallskip

The proof of our main result is based on comparison arguments, in which we control the solution $v^\delta$ of the approximate cell problem
\begin{equation} \label{e.approxcellprob}
\delta v^\delta + \left( \left| p+Dv^\delta \right|^2 - 1 \right)^2 - V(y) = 0 \quad \mbox{in} \ \Rd
\end{equation}
by the maximal subsolutions of the following family of ``sub-equations"
\begin{equation} \label{e.subequation}
\left|Du\right|^2 = 1 + \sigma \sqrt{\mu + V(y) },
\end{equation}
where the real parameters $\mu$ and $\sigma$ range over $-\inf V \leq \mu < \infty$ and $\sigma \in [-1,1]$. Notice that~\eqref{e.subequation} for $\sigma = \pm 1$ can be formally derived from the \emph{metric problem} associated to~\eqref{e.pde}, which is roughly
\begin{equation} \label{e.metricproblem}
\left( \left| Du \right|^2 - 1 \right)^2 - V(y) = \mu,
\end{equation}
by taking the square root of~\eqref{e.metricproblem}. As it turns out that we must consider~\eqref{e.subequation} also for $-1<\sigma<1$, in order to ``connect" the branches of the square root function. The key insight is that, while the solutions of both~\eqref{e.subequation} and~\eqref{e.metricproblem} have a subadditive structure and thus deterministic limits by the ergodic theorem, there is more information contained in the former than the latter. Indeed, as we show, there is just enough information in~\eqref{e.subequation} to allow us to deduce that~\eqref{e.pde} homogenizes. 

\smallskip

The method we introduce here is applicable to somewhat more general nonconvex equations than~\eqref{e.pde} and, in particular, applies to any equation of the form
\begin{equation} \label{e.moregen}
u_t^\ep + \Phi\!\left(K\!\left(Du^\ep\right) \right) - V\left( \frac x\ep \right) = 0,
\end{equation}
where $\Phi:\R \to \R$ is continuous, $\Phi(s) \rightarrow +\infty$ as $s\to \infty$, the function $K:\Rd \to \R$ is convex and $K(p) \rightarrow +\infty$ as $|p| \to \infty$. In dimension $d=1$, the geometry allows us to take $K(p) =p$ (even though this is not coercive) and we therefore get a general result for any coercive Hamiltonian which is the sum of a deterministic, coercive energy profile and a random potential. Our arguments do not require the dependence of the equation on the gradient variable to be radial or even. The reason we focus on~\eqref{e.pde} rather than~\eqref{e.moregen} is because all of the major difficulties are encountered in the analysis of the former, and that of the latter leads to more complicated notation and bookkeeping issues which distract from the main points. Since~\eqref{e.moregen} is far from a complete class of equations,  the problem of homogenizing general coercive, possibly nonconvex Hamilton-Jacobi equations remains open.

\subsection{Precise statement of the main result}
The random potential is modeled by a probability measure on the set of all potentials. To make this precise, we take~
\begin{equation*} \label{}
\Omega:=\BUC(\Rd)
\end{equation*}
to be the space of real-valued, bounded and uniformly continuous functions on~$\Rd$. We define~$\F$ to be the~$\sigma$-algebra on~$\Omega$ generated by pointwise evaluations, that is
\begin{equation*} \label{}
\F := \, \mbox{$\sigma$--algebra generated by the family of maps} \quad \left\{ V \mapsto V(x) \,:\,  x\in \Rd \right\}. 
\end{equation*}
The translation group action of $\Rd$ on $\Omega$ is denoted by $\{ T_y\}_{y\in \Rd}$, that is, $T_y:\Omega \to \Omega$ is defined by
\begin{equation*} \label{}
\left( T_y V\right)(x) := V(x+y). 
\end{equation*}
We consider a probability measure $\P$ on $(\Omega,\F)$ satisfying the following properties: there exists $K_0>0$ such that 
\begin{equation} \label{e.pub}
\P \left[ \sup_{x\in \Rd} \left| V(x) \right| \leq K_0\right] =1  \quad \mbox{(uniform boundedness),}
\end{equation}
for every $E \in \F$ and $y\in \Rd$,
\begin{equation} \label{e.pstat}
\P \left[ E \right] = \P \left[ T_yE \right] \quad \mbox{(stationarity)}
\end{equation}
and
\begin{equation} \label{e.perg}
\P \big[ \cap_{z\in \Rd} T_zE \big]  \in \{ 0,1 \} \quad \mbox{(ergodicity).}
\end{equation}

We now present the main result. Recall that, for each~$\ep > 0 $ and~$g\in \BUC(\Rd)$, there exists a unique solution~$u^\ep(\cdot,g)\in C(\Rd \times[0,\infty))$ of~\eqref{e.pde} in~$\Rd \times (0,\infty)$, subject to the initial condition $u^\ep(x,0,g) = g(x)$. All differential equations and inequalities in this paper are to be interpreted in the viscosity sense (see~\cite{EBook}). 

\begin{theorem}
\label{t.main}
Assume~$\P$ is a probability measure on~$(\Omega,\F)$ satisfying~\eqref{e.pub},~\eqref{e.pstat} and~\eqref{e.perg}. Then there exists $\overline H \in C(\Rd)$ satisfying
\begin{equation} \label{e.Hbarcoer}
\overline H(p) \rightarrow +\infty \quad \mbox{as} \ |p| \to \infty
\end{equation}
such that, if we denote, for each $g\in \BUC(\Rd)$, the unique solution of~\eqref{e.pdehom} subject to the initial condition $u(x,0) = g(x)$ by $u(x,t,g)$, then
\begin{equation*} \label{}
\P \left[ \forall g\in \BUC(\Rd), \ \forall k>0, \ \limsup_{\ep \to 0} \sup_{(x,t) \in B_{k} \times [0,k]} \left| u^\ep(x,t,g) - u(x,t,g) \right| = 0 \right] = 1. 
\end{equation*}
\end{theorem}

Some qualitative properties of $\overline H$, including the confirmation that its basic shape resembles that of Figure~\ref{fig.effham1}, are presented in Section~\ref{ss.Hbar}. 

\subsection{Outline of the paper}
In Section~\ref{ss.subeq}, we introduce the maximal subsolutions of~\eqref{e.subequation}, study their relationship to~\eqref{e.metricproblem} and show that they homogenize. We construct $\overline H$ in Section~\ref{ss.Hbar} and study some of its qualitative features. The proof of~Theorem~\ref{t.main} is the focus of~Section~\ref{s.homog}, where we compare the maximal subsolutions of~\eqref{e.subequation} to the solutions of~\eqref{e.approxcellprob}.

\section{Identification of the effective Hamiltonian}

Following the metric problem approach to homogenization introduced in~\cite{AS2}, one is motivated to consider, for~$\mu \in \R$, maximal subsolutions of the  equation
\begin{equation} \label{e.naiveMP}
\left( \left| D u \right|^2 - 1 \right)^2 - V( y ) = \mu \quad \mbox{in} \ \Rd.
\end{equation}
Unfortunately, unlike the convex setting, it turns out (as is well-known) that the maximal subsolutions of~\eqref{e.naiveMP} do not encode enough information to identify~$\overline H$, much less prove homogenization. This is not surprising since, by the subadditive nature of the maximal subsolutions, if they could identify~$\overline H$ then the latter would necessarily be convex. Instead, we consider maximal subsolutions of the ``sub-equation"
\begin{equation} \label{e.subpde}
\left|Du\right|^2 = 1 + \sigma \sqrt{\mu + V(y) }\quad \mbox{in} \ \Rd,
\end{equation}
with we take the parameters~$\mu \geq -\inf_{\Rd} V$ and~$\sigma\in [-1,1]$. The idea is that we can control solutions of~\eqref{e.pde} by the maximal subsolutions of~\eqref{e.subpde}, varying the parameters~$\mu$ and~$\sigma$ in an appropriate way. Observe that we may formally derive~\eqref{e.subpde} with~$\sigma=\pm1$ from~\eqref{e.naiveMP} by taking the square root of the equation. 

\smallskip

\subsection{The maximal subsolutions of~\eqref{e.subpde}}
\label{ss.subeq}
We define the maximal subsolutions of~\eqref{e.subpde} and review their deterministic properties. Throughout this subsection we suppose for convenience that
\begin{equation} \label{e.infVzero}
\inf_{\Rd} V = 0.
\end{equation}
For every $\mu \geq 0$, $-1 \leq \sigma \leq 1$ and $z\in \Rd$, we define
\begin{equation} \label{e.mmudef}
m_{\mu,\sigma}(y,z) := \sup\left\{ u(y) - u(z) \,:\, u \in \USC(\Rd) \ \mbox{is a subsolution of~\eqref{e.subpde}}  \right\}.
\end{equation}
Clearly this definition is void if~\eqref{e.subpde} possesses no subsolutions, which occurs if and only if the right-hand side is not nonnegative, that is, if and only if
\begin{equation*} \label{}
\sigma\left(  \mu + \sup_{\Rd} V \right)^{1/2} < -1.
\end{equation*}
In this case, we simply take $m_{\mu,\sigma} \equiv -\infty$. Otherwise, we note that $m_{\mu,\sigma} \geq 0$. 

\smallskip

In the next proposition, we summarize some basic properties of~$m_{\mu,\sigma}$ and relate it to the equation
\begin{equation} \label{e.cousin}
\left( \left| Dw \right|^2 - 1 \right)^2 = \sigma^2 \left( \mu+V(y)\right). 
\end{equation}
Note that~\eqref{e.cousin} is the same as~\eqref{e.naiveMP} in the case that $\sigma \in \{ -1, 1\}$.

\begin{proposition}
\label{p.mmurho}
Fix $V \in \Omega$ satisfying~\eqref{e.infVzero}, $\mu \geq 0$ and $\sigma \in [-1,1]$ such that 
\begin{equation} \label{e.admin}
\sigma \left(  \mu + \sup_{\Rd} V \right)^{1/2} \geq -1.
\end{equation}
For every $y,z\in \Rd$, 
\begin{equation} \label{e.mmusym}
m_{\mu,\sigma}(y,z) = m_{\mu,\sigma}(z,y).
\end{equation}
For every $x,y,z\in \Rd$,
\begin{equation} \label{e.subadd}
m_{\mu,\sigma}(y,z) \leq m_{\mu,\sigma}(y,x) + m_{\mu,\sigma}(x,z).
\end{equation}
For every $z\in \Rd$, $m_\mu(\cdot,z) \in C^{0,1}(\Rd)$ and 
\begin{equation} \label{e.sub}
-m_{\mu,\sigma}(\cdot,z) \quad \mbox{is a subsolution of~\eqref{e.cousin} in}\quad \Rd \setminus \{ z \},
\end{equation}
\begin{equation} \label{e.super}
m_{\mu,\sigma}(\cdot,z) \quad \mbox{is a supersolution of~\eqref{e.cousin} in}\quad \Rd \setminus \{ z \}
\end{equation}
and, moreover,
\begin{equation} \label{e.super.global}
\mbox{if} \ \sigma \leq 0, \   \mbox{then} \quad  m_{\mu,\sigma}(\cdot,z) \quad \mbox{is a supersolution of~\eqref{e.cousin} in}\quad \Rd.
\end{equation}

\end{proposition}
\begin{proof}
Since~\eqref{e.subpde} is a convex equation, a function~$u\in \USC(\Rd)$ is a subsolution of~\eqref{e.subpde} if and only if~$u\in C^{0,1}_{\mathrm{loc}}(\Rd)$ (and thus $u$ is differentiable almost everywhere) and~$u$ satisfies~\eqref{e.subpde} at almost every point of~$\Rd$. See e.g.~\cite{BJ} or~\cite[Lemma 2.1]{AS2}. Since $V$ is uniformly bounded, a subsolution must in fact be globally Lipschitz, i.e., $u\in C^{0,1}(\Rd)$. Thus, for each $z\in \Rd$, $m_{\mu,\sigma}(\cdot,z)$ is the supremum of a family of equi-Lipschitz functions on $\Rd$ and hence belongs to $C^{0,1}(\Rd)$. As $u$ is the supremum of a family of subsolutions of~\eqref{e.subpde}, we have
\begin{equation} \label{e.subsub}
m_{\mu,\sigma}(\cdot,z) \quad \mbox{is a subsolution of~\eqref{e.subpde} in}\quad \Rd.
\end{equation}
We also obtain from the above characterization of subsolution of~\eqref{e.subpde} that $u \in \USC(\Rd)$ is a subsolution of~\eqref{e.subpde} if and only if $-u$ is also a subsolution of~\eqref{e.subpde}. This together with the definition of $m_{\mu,\sigma}$ yields~\eqref{e.mmusym} as well as that
\begin{equation} \label{e.subsubneg}
-m_{\mu,\sigma}(\cdot,z) \quad \mbox{is a subsolution of~\eqref{e.subpde} in}\quad \Rd.
\end{equation}
Finally, by the maximality of $m_{\mu,\sigma}(\cdot,z)$, the Perron method yields that
\begin{equation} \label{e.supersub}
m_{\mu,\sigma}(\cdot,z) \quad \mbox{is a supersolution of~\eqref{e.subpde} in}\quad \Rd \setminus\{ z \}.
\end{equation}
A proof of~\eqref{e.supersub} can also be found in~\cite[Proposition~3.2]{AS2}.

\smallskip

The subadditivity~\eqref{e.subadd} of $m_{\mu,\sigma}$ is immediate from maximality. Indeed, since $m_{\mu,\sigma}(\cdot,z) - m_{\mu,\sigma}(x,z)$ is a subsolution of~\eqref{e.subpde} in $\Rd$, we may use it as an admissible function in the definition of $m_{\mu,\sigma}(y,x)$. This yields~\eqref{e.subadd}.

\smallskip

Proceeding with the demonstration of~\eqref{e.sub},~\eqref{e.super} and~\eqref{e.super.global}, we select a smooth test function $\phi \in C^{\infty}(\Rd)$ and $x_0 \in \Rd$ such that 
\begin{equation} \label{e.test1}
y\mapsto m_{\mu,\sigma}(y,z) - \phi(y) \quad \mbox{has a local minimum at} \ y=y_0
\end{equation}
which is equivalent to
\begin{equation} \label{e.test2}
y\mapsto -m_{\mu,\sigma}(y,z) - (-\phi(y)) \quad \mbox{has a local maximum at} \ y=y_0.
\end{equation}
According to~\eqref{e.subsubneg} and~\eqref{e.test2},
\begin{equation} \label{e.testphi1}
\left|D\phi(y_0) \right|^2 \leq 1 + \sigma \sqrt{\mu + V(y_0) }.
\end{equation}
If $\sigma \leq 0$, then~\eqref{e.testphi1} implies that 
\begin{equation*} \label{}
\left( \left|D\phi(y_0) \right|^2 - 1 \right)^2 \geq \sigma^2  \left( \mu + V(y_0) \right). 
\end{equation*}
This completes the proof of~\eqref{e.super.global}. If $y_0 \neq z$, then~\eqref{e.supersub} and~\eqref{e.test1} yield 
\begin{equation*} 
\left|D\phi(y_0) \right|^2 \geq 1 + \sigma \sqrt{\mu + V(y_0) }
\end{equation*}
which, together with~\eqref{e.testphi1}, gives
\begin{equation*} 
\left|D\phi(y_0) \right|^2 = 1 + \sigma \sqrt{\mu + V(y_0) }.
\end{equation*}
Rearranging the equation and squaring the previous line, we get
\begin{equation} \label{e.testphi2}
\left( \left|D\phi(y_0) \right|^2 - 1 \right)^2  = \sigma^2 \left( \mu + V(y_0) \right).
\end{equation}
In view of the fact that~\eqref{e.test1} and~\eqref{e.test2} are equivalent, and that~\eqref{e.testphi2} is symmetric in $\phi$ and $-\phi$, we have proved both~\eqref{e.sub} and~\eqref{e.super}. 
\end{proof}

\subsection{Limiting shapes of $m_{\mu,\sigma}$ and identification of $\overline H$}
\label{ss.Hbar}

Since $m_{\mu,\sigma}$ is defined to be maximal  , the subadditive ergodic theorem implies that $m_{\mu,\sigma}$ is deterministic in the rescaled macroscopic limit. The precise statement we need is summarized in Proposition~\ref{p.shape}. Before presenting it, we first observe that $\inf_{\Rd} V$ and $\sup_{\Rd} V$ are deterministic quantities, thanks to the ergodicity hypothesis. 

\begin{lemma}
\label{l.detm}
There exist $\overline v,\underline v\in \R$ such that 
\begin{equation*} \label{}
\P\left[ \sup_{\Rd} V = \overline v \right] = \P\left[ \inf_{\Rd} V = \underline v \right] = 1.
\end{equation*}
\end{lemma}
\begin{proof}
For each $t\in \R$, the events $\left\{ V \in \Omega\,:\, \inf_{\Rd} V < t \right\}$ and $\left\{ V \in \Omega\,:\, \sup_{\Rd} V > t \right\}$ are invariant under translations and therefore have probability either 0 or 1 by~\eqref{e.perg}. 
We take $\overline v$ to be the largest value of $t$ for which $\P \left[\sup_{\Rd} V > t \right] = 1$ and $\underline v$ to be the smallest value of $t$ for which $\P \left[\inf_{\Rd} V < t \right] = 1$. In view of~\eqref{e.pub}, we have $-K_0\leq \underline v \leq \overline v \leq K_0$. The statement of the lemma follows. 
\end{proof}

We assume throughout the rest of the paper that~$\underline v = 0$. Note that we may, without loss of generality, subtract any constant we like from the random potential~$V$ without altering the statement of~Theorem~\ref{t.main}.

\smallskip

The value of $\overline v$ prescribes, almost surely, the set of parameters $(\mu,\sigma)$ for which~$m_{\mu,\sigma}$ is finite, i.e., for which~\eqref{e.admin} holds. We denote this by
\begin{equation} \label{e.admindef}
\A:= \left\{ (\mu,\sigma) \in [0,\infty) \times [-1,1] \,:\, \sigma(\mu+\overline v)^{1/2} \geq -1 \right\}.
\end{equation}
It is convenient to set
\begin{equation} \label{e.kappa}
\kappa:= 1-\overline v.
\end{equation}
Note that if $\kappa \geq 0$, then $(\kappa,-1) \in \A$ and $\kappa$ is the largest value of $\mu$ for which $(\mu,-1) \in \A$. If on the other hand $\kappa < 0$, then $(\mu,-1) \not\in\A$ for every $\mu \geq 0$.

We also define the following subset $\A'$ of $\A$, which consists of those parameters which play a role in the proof of~Theorem~\ref{t.main}:
\begin{equation} \label{e.Aprime}
\A':= \big\{ (\mu,\sigma) \in \A\,:\, \mu = 0 \ \mbox{or} \ \sigma \in \{ -1,1\} \big\}
\end{equation}
Observe that there exists a unique element $(\mu_*,\sigma_*)\in \A'$ for which 
\begin{equation*} \label{}
\sigma_*\left( \mu_* + \overline v \right)^{1/2} = -1. 
\end{equation*}
In fact, with $\kappa$ as above, we have
\begin{equation} \label{e.minelem}
(\mu_*,\sigma_*) = \left\{ \begin{aligned} 
& (\kappa,-1) && \mbox{if} \ \kappa \geq 0, \\
& (0,-\overline v^{\,-1/2}) && \mbox{otherwise}.
\end{aligned} \right.
\end{equation}

\smallskip

We next establish some simple bounds on the growth of $m_{\mu,\sigma}$.

\begin{lemma}
\label{l.easybnds}
Assume $(\mu,\sigma)\in \A$ and $V\in \Omega$ satisfies $\inf_{\Rd} V = 0$ and $\sup_{\Rd} V = \overline v$. Then, for every $y,z\in \Rd$, we have the following: in the case that $\sigma \leq 0$,
\begin{equation} \label{e.mmuyz.-1}
\left( 1 +\sigma (\mu + \overline v)^{1/2} \right)^{1/2} |y-z| \leq m_{\mu,\sigma}(y,z) \leq \left( 1 +\sigma \mu^{1/2} \right)^{1/2} |y-z|
\end{equation}
and, in the case that $\sigma \geq 0$, 
\begin{equation} \label{e.mmuyz.+1}
 \left( 1 +\sigma \mu^{1/2} \right)^{1/2} |y-z|\leq m_{\mu,\sigma}(y,z) \leq \left( 1 +\sigma (\mu + \overline v)^{1/2} \right)^{1/2} |y-z|.
\end{equation}
\end{lemma}
\begin{proof}
The arguments for~\eqref{e.mmuyz.-1} and~\eqref{e.mmuyz.+1} are almost the same, so we only give the proof of~\eqref{e.mmuyz.-1}. The lower bound is immediate from the definition of $m_{\mu,\sigma}(\cdot,z)$ and the fact that the left side of~\eqref{e.mmuyz.-1}, as a function of $y$, is a subsolution of~\eqref{e.subpde} in~$\Rd$. To get the upper bound, we observe that any subsolution $u \in \USC(\Rd)$ of~\eqref{e.subpde} satisfies
\begin{equation} \label{e.uppbndgrbg}
|Du|^2 \leq \left( 1 +\sigma \mu^{1/2} \right) \quad \mbox{in} \ \Rd. 
\end{equation}
In particular, by the characterization of subsolutions mentioned in the first paragraph of the proof of~Proposition~\ref{p.mmurho}, we deduce that~\eqref{e.uppbndgrbg} holds at almost every point of~$\Rd$. This implies that~$u$ is Lipschitz with constant $\left( 1 +\sigma \mu^{1/2} \right)^{1/2}$. This argument applies to~$m_{\mu,\sigma}(\cdot,z)$ by~\eqref{e.subsub}. Since $m_{\mu,\sigma}(z,z) = 0$, we obtain the upper bound of~\eqref{e.mmuyz.-1}.
\end{proof}

We next prove some continuity and monotonicity properties for the function $(\mu,\sigma) \mapsto m_{\mu,\sigma}(y,z)$ on $\A$.

\begin{lemma}
\label{l.contmono}
Fix $V\in \Omega$ for which $\inf_{\Rd} V= 0$ and $\sup_{ \Rd}V = \overline v$ and suppose that $(\mu,\sigma) \in \A$ is such that $\sigma (\mu+\overline v)^{1/2} > -1$. Then 
\begin{equation} \label{e.cont}
\lim_{\A \ni (\nu,\tau) \to (\mu,\sigma)} \ \sup_{y,z\in\Rd,\, y\neq z} \frac{\left| m_{\mu,\sigma}(y,z) - m_{\nu,\tau}(y,z) \right|}{|y-z|} = 0. 
\end{equation}
For every pair $(\mu,\sigma) , (\nu,\tau) \in \A$ and $y,z\in \Rd$, we have
\begin{equation} \label{e.mono}
m_{\mu,\sigma}(y,z) \leq m_{\nu,\tau}(y,z) \qquad \mbox{provided that} \qquad \left\{ \begin{aligned} 
& \mu = \nu \quad \mbox{and} \quad \sigma \leq \tau, \\ & \qquad \mbox{or} \\
& \sigma = \tau \quad \mbox{and} \quad \sigma \mu \leq \sigma \nu.
\end{aligned} \right.
\end{equation}
Moreover, for every $(\mu,\sigma), (\nu,\tau) \in \A$ with $\sigma \mu < \tau \nu$, there exists $c>0$ such that, for all $y,z\in \Rd$,
\begin{equation} \label{e.mono-strict}
m_{\mu,\sigma}(y,z) \leq m_{\nu,\tau}(y,z) - c|y-z|.
\end{equation}
\end{lemma}
\begin{proof}
Let $0< \ep < 1$, and observe that, by~\eqref{e.subsub}, for $\lambda:= 1-\ep$, the function $w:= \lambda m_{\mu,\sigma}(\cdot,z)$ is a subsolution of the equation
\begin{equation*} \label{}
|Dw|^2 \leq \lambda^2 \left( 1 + \sigma \sqrt{ \mu +V(y)} \right) \quad \mbox{in} \ \Rd. 
\end{equation*}
Observe that the infimum over $\Rd$ of the term in parentheses on the right-hand side is positive by assumption. Thus if $(\nu,\tau)$ is sufficiently close to $(\mu,\sigma)$, we have that for all $y\in \Rd$,
\begin{equation*} \label{}
\lambda^2 \left( 1 + \sigma \sqrt{ \mu +V(y)} \right) < 1 + \tau \sqrt{ \nu +V(y)}.
\end{equation*}
By maximality, we deduce that $w \leq m_{\nu,\tau}(\cdot,z)$ for all $(\nu,\tau)$ sufficiently close to $(\mu,\sigma)$, depending on~$\ep$. According to the bounds in Lemma~\ref{l.easybnds}, we obtain, for a constant $C>0$ depending only on $(\mu,\sigma,\overline v)$, the estimate
\begin{equation*} \label{}
m_{\mu,\sigma}(y,z) \leq m_{\nu,\tau}(y,z) + C\ep|y-z|
\end{equation*}
Reversing the roles of $(\mu,\sigma)$ and $(\nu,\tau)$, using that $\tau(\nu+\overline v) >-1$ for $(\nu,\tau)$ close to $(\mu,\sigma)$, and arguing similarly, we get, for all $(\nu,\tau)$ sufficiently close to $(\mu,\sigma)$, that
\begin{equation*} \label{}
m_{\nu,\tau}(y,z) \leq m_{\mu,\sigma}(y,z) + C\ep|y-z|.
\end{equation*}
This completes the proof of~\eqref{e.cont}.

The monotonicity property~\eqref{e.mono} is immediate from the definition~\eqref{e.mmudef} since the condition on the right of~\eqref{e.mono} implies that the right side of~\eqref{e.subpde} is larger for $(\nu,\tau)$ than for $(\mu,\sigma)$, and hence the admissible class of subsolutions in~\eqref{e.mmudef} is larger.

The strict monotonicity property in the last statement of the lemma follows from the fact, which is easy to check from the characterization of subsolutions mentioned in the proof of Proposition~\ref{p.mmurho}, that $y \mapsto m_{\mu,\sigma}(y,z) + c|y-z|$ is a subsolution of 
\begin{equation*} \label{}
|Dw|^2 \leq 1 + \sigma \sqrt{\nu + V(y)} \quad  \mbox{in} \ \Rd,
\end{equation*}
provided $c>0$ is sufficiently small, depending on a lower bound for $\sigma(\nu-\mu)$.
\end{proof}

The following proposition is a special case of, for example,~\cite[Proposition 4.1]{AS2} or~\cite[Proposition 2.5]{AT1}), and so we do not present the proof. The argument is an application of the subadditive ergodic theorem, using the subadditivity of $m_{\mu,\sigma}$,~\eqref{e.subadd}.

\begin{proposition}
For each $(\mu,\sigma)\in \A$, there exists a convex, positively homogeneous function $\overline m_{\mu,\sigma} \in C(\Rd)$ such that 
\label{p.shape}
\begin{equation*} \label{}
\P \left[ \forall (\mu,\sigma) \in \mathcal A, \ \forall R>0, \ \limsup_{t\to \infty} \sup_{y,z\in B_{R}} \left| \frac{m_{\mu,\sigma}(ty,tz)}{t} - \overline m_{\mu,\sigma}(y-z)  \right| = 0 \right]  = 1.
\end{equation*}
\end{proposition}

We are now ready to construct $\overline H$. We continue by introducing two functions
\begin{equation*} \label{}
\overline H^{\,-} : \Rd \to \{ -\infty \} \cup [0,\infty) \quad \mbox{and} \quad \overline H^{\,+}: \Rd \to [0,\infty). 
\end{equation*}
defined by 
\begin{align*}
\overline H^{\,-}\!(p) & := \sup\left\{ \mu \geq 0 \,:\, \forall y\in \Rd, \ \overline m_{\mu,-1} (y) \geq p\cdot y  \right\},\\
\overline H^{\,+}\!(p) & := \inf\left\{ \mu \geq 0 \,:\, \forall y\in \Rd, \ \overline m_{\mu,+1}(y) \geq p\cdot y  \right\}.
\end{align*}
We take $\overline H^{\,-}\!(p):= -\infty$ if the admissible set in its definition  is empty. Since $\mu \mapsto \overline m_{\mu,-1}(\cdot)$ is decreasing, we see that $\overline H^-(p) = -\infty$ if and only if there exists $y \in \Rd$ such that $\overline m_{0,-1}(y) < p\cdot y$. We define the effective Hamiltonian to be the maximum of these:
\begin{equation*} \label{}
\overline H(p) := \max\left\{ \overline H^{\,-}\!(p) , \overline H^{\,+}\!(p) \right\}. 
\end{equation*}
Observe that since, for all $\mu,\nu\geq 0$,
\begin{equation*} \label{}
\overline m_{\mu,-1} \leq \overline m_{\nu,1},
\end{equation*}
we have that 
\begin{equation*} \label{}
\left\{ p\in \Rd \,:\, \overline H^{\,+}\!(p) > 0 \right\} \subseteq \left\{ p\in \Rd \,:\, \overline H^{\,-}\!(p)= -\infty \right\}.
\end{equation*}
Therefore we can also write
\begin{equation*} \label{}
\overline H(p)  =  \left\{ \begin{aligned} 
& \overline H^{\,-}\!(p) && \mbox{if} \ \overline H^{\,-}\!(p) \neq -\infty,\\
& \overline H^{\,+}\!(p) && \mbox{otherwise}.
\end{aligned} \right.
\end{equation*}

We next check that $\overline H$ is coercive, i.e., that~\eqref{e.Hbarcoer} holds.

\begin{lemma}
\label{l.coercivity}
For every $p\in \Rd$,
\begin{equation} \label{e.Hbarcoercive}
\left( |p|^2 -1 \right)^2 - \overline v \leq \overline H(p) \leq \left( |p|^2 - 1\right)^2.
\end{equation}
\end{lemma}
\begin{proof}
According to Proposition~\ref{p.shape} and Lemmas~\ref{l.detm} and~\ref{l.easybnds}, for every $\mu\geq 0$,
\begin{equation*} \label{}
\left( 1 - (\mu + \overline v)^{1/2} \right) |y| \leq \overline m_{\mu,-1}(y) \leq \left( 1 - \mu^{1/2} \right) |y|,
\end{equation*}
provided that $(\mu,-1)\in\A$, and
\begin{equation*} \label{}
\left( 1 + \mu^{1/2} \right) |y| \leq \overline m_{\mu,+1}(y) \leq \left( 1 + (\mu + \overline v)^{1/2} \right) |y|.
\end{equation*}
In view of the definition of $\overline H$, this yields the estimate~\eqref{e.Hbarcoercive}.
\end{proof}

In order to describe $\overline H$ further, we partition $\Rd$ into four regions, generally corresponding to the following features in the graph of $\overline H$: the flat hilltop, the flat valley, the slope between the latter two, and the unbounded region outside the flat valley (see Figure~\ref{fig.effham1}). We define

\begin{equation} \label{e.Ki}
\left\{\begin{aligned}
K_1 & := \bigcap_{(\mu,\sigma) \in \A' \setminus\{  (\mu_*,\sigma_*)\} } \partial \overline m_{\mu,\sigma}(0),    &  K_2 & := \bigcup_{(\mu,-1) \in \A', \, 0< \mu < \mu_*} \partial \overline m_{\mu,-1}(\partial B_1),  \\
K_3 & := \bigcup_{(0,\sigma) \in  \A' \setminus\{  (\mu_*,\sigma_*)\}} \partial \overline m_{0,\sigma}(\partial B_1),  & K_4 & := \bigcup_{\mu > 0}  \partial \overline m_{\mu,1}(\partial B_1).
\end{aligned}\right.
\end{equation}
Here $\partial \phi(x_0)$ denotes the subdifferential of a convex function $\phi:\Rd \to \R$ at~$x_0\in \Rd$,
\begin{equation*} \label{}
\partial \phi(x_0) := \left\{ q\in\Rd\,:\, \phi(y) \geq \phi(x_0) + q \cdot (y-x_0) \right\},
\end{equation*}
and we write $\partial \phi(E) := \cup\left\{ \partial \phi(x)\,:\, x\in \E\right\}$ for $E\subseteq \Rd$.

We remark that~$0 \in K_1$ by the nonnegativity of~$m_{\mu,\sigma}$  and~$K_2 = \emptyset$ if and only if~$\mu_* = 0$. Since $m_{\mu,0}(y,0) =  \overline m_{\mu,0}(y) = |y|$ for every $\mu\geq 0$, we see that $\partial B_1 \subseteq K_3$. Finally, we note that~$K_4$ is unbounded, while~$K_1\cup K_2 \cup K_3$ is bounded. 

\smallskip

The following proposition gives us a representation of $\overline H$ which is convenient for the proof of Theorem~\ref{t.main}. It also confirms that the basic features of $\overline H$ are portrayed accurately in Figure~\ref{fig.effham1}.

\begin{proposition}
\label{p.partition}
For each $p\in \Rd \setminus K_1$, there exists a unique $\mu \geq 0$ such that, for some $(\mu,\sigma) \in \A'$, we have $p \in \partial \overline m_{\mu,\sigma}(\partial B_1)$. In particular,~$\{ K_1,K_2,K_3,K_4\}$ is a disjoint partition of~$\Rd$. Moreover, with $\mu_*$ as defined in~\eqref{e.minelem}, we have
\begin{equation} \label{e.Hbargoodform}
\overline H(p) = \left\{ \begin{aligned}
& \mu_*  && \mbox{for} \ p\in K_1,\\
& \mu && \mbox{for} \ p \in \partial \overline m_{\mu,\sigma}(\partial B_1),\ (\mu,\sigma) \in \A'.
 \end{aligned} \right.
\end{equation}
\end{proposition}
\begin{proof}
We move along the path $\A'$ starting at $(\mu_*,\sigma_*)$. If $\mu_*>0$ and hence $\sigma_*=-1$, then we move in straight line segments from $(\mu_*,-1)$ to $(0,-1)$ to $(0,1)$ to $(\infty,1)$; otherwise, if $\mu_*=0$, then we move first from $(0,\sigma_*)$ to $(0,1)$ and then to $(\infty,1)$. 

By Lemma~\ref{l.contmono}, the graph of the positively homogeneous, convex function $\overline m_{\mu,\sigma}$ is continuous and increasing as we move along the path. Therefore, given $p\in \Rd\setminus K_1$, we can stop at the first point $(\mu,\sigma)\in \A'\setminus \{ (\mu_*,\sigma_*) \}$ in the path at which the graph of $\overline m_{\mu,\sigma}$ is tangent to that of the plane $y\mapsto p\cdot y$. Indeed, $p\not\in K_1$ ensures that the plane $p\cdot y$ is not below the graph of $\overline m_{\mu,\sigma}$ for every $(\mu,\sigma)\in \A'\setminus \{ (\mu_*,\sigma_*) \}$, and this point must be reached at or before $((|p|^2-1)^2,1)$, by the estimate~\eqref{e.mmuyz.+1}. The uniqueness of $\mu$ follows from the last statement of Lemma~\ref{l.contmono}. This completes the proof of the first statement. The formula~\eqref{e.Hbargoodform} is then immediate from the definition of $\overline H$ and Lemma~\ref{l.contmono}. 
\end{proof}

\section{Proof of homogenization}
\label{s.homog}

We consider, for each $p\in \Rd$ and $\delta > 0$, the \emph{approximate cell problem}
\begin{equation} \label{e.appcell}
\delta v^\delta + \left( \left| p+Dv^\delta \right|^2 - 1 \right)^2 - V(y) = 0 \quad \mbox{in} \ \Rd. 
\end{equation}
It is classical that, for every $p\in\Rd$ and $\delta > 0$, there exists a unique viscosity solution~$v^\delta=v^\delta(\cdot,p) \in C(\Rd)$ of~\eqref{e.appcell} subject to the growth condition
\begin{equation*} \label{}
\limsup_{|y| \to \infty} \frac{v^\delta(y)}{|y|} = 0. 
\end{equation*}
In fact, by comparing $v^\delta(\cdot,p)$ to constant functions we immediately obtain that $v^\delta(\cdot,p)$ is bounded and 
\begin{equation} \label{e.dvdbnd}
-\frac1\delta \left( (|p|^2 -1)^2 - \inf_{\Rd} V(y) \right) \leq v^\delta(\cdot,p) \leq -\frac1\delta \left( (|p|^2 -1)^2 - \sup_{\Rd} V(y) \right).
\end{equation}
It follows from~\eqref{e.dvdbnd} and the coercivity of the equation that $v^\delta$ is Lipschitz and, for $C>0$ depending only on an upper bound for $|p|$ and $\sup_{\Rd} V$, we have
\begin{equation} \label{e.dvdLip}
\sup_{\Rd} \left|Dv^\delta(\cdot,p)\right| \leq C
\end{equation}
Using~\eqref{e.dvdLip} and comparing $v^\delta(\cdot,p)$ to $v^\delta(\cdot,q) \pm C\delta^{-1}|p-q|$, we obtain, for a constant $C> 0$ depending only on an upper bound for $\max\{ |p|,|q| \}$ and $\sup_{\Rd} V$, the estimate
\begin{equation} \label{e.dvdcontp}
\sup_{\Rd} \left| \delta v^\delta(\cdot,p) - \delta v^\delta(\cdot,q) \right|  \leq C|p-q|. 
\end{equation}

\smallskip

By the perturbed test function method, Theorem~\ref{t.main} can be reduced to the following proposition. 

\begin{proposition}
\label{p.cell}
\begin{equation} \label{e.cell}
\P \left[ \forall R>0, \ \limsup_{\delta \to 0} \sup_{p\in B_R} \sup_{B_{R/\delta}} \left| \delta v^\delta(\cdot,p) + \overline H(p) \right| =0 \right] = 1. 
\end{equation}
\end{proposition}

We omit the demonstration that Proposition~\ref{p.cell} implies Theorem~\ref{t.main}, since it is classical and can also be obtained for example by applying~\cite[Lemma 7.1]{ACS}. The argument for Proposition~\ref{p.cell} is broken into the following five lemmas. Recall that $\A$ is the set of admissible parameters $(\mu,\sigma)$ defined in~\eqref{e.admindef}.

\begin{lemma}
\label{l.below.+1}
\begin{equation*} \label{}
\P \bigg[ \forall (\mu,1) \in \mathcal A,\ \forall p\in \partial \overline m_{\mu,1}(\partial B_1), \ \liminf_{\delta \to 0} \, -\delta v^\delta (0,p) \geq \mu \bigg]  = 1.
\end{equation*}
\end{lemma}

\begin{lemma}
\label{l.below.-1}
\begin{equation*} \label{}
\P \bigg[ \forall (\mu,-1) \in \mathcal A,\ \forall p\in \partial \overline m_{\mu,-1}(0), \ \liminf_{\delta \to 0} \, -\delta v^\delta (0,p) \geq \mu \bigg]  = 1.
\end{equation*}
\end{lemma}

\begin{lemma}
\label{l.below-valley}
\begin{equation*} \label{}
\P \bigg[ \forall p\in \Rd, \ \liminf_{\delta \to 0} \, -\delta v^\delta (0,p) \geq 0  \bigg]  = 1.
\end{equation*}
\end{lemma}

\begin{lemma}
\label{l.above}
\begin{equation*} \label{}
\P \bigg[ \forall (\mu,\sigma) \in \mathcal A, \  \forall p\in \partial \overline m_{\mu,\sigma}(\partial B_1), \ \limsup_{\delta \to 0} \, -\delta v^\delta (0,p) \leq \mu \bigg]  = 1.
\end{equation*}
\end{lemma}

\begin{lemma}
\label{l.above-hilltop}
\begin{equation*} \label{}
\P \bigg[ \forall p\in B_1, \ \limsup_{\delta \to 0} \, -\delta v^\delta (0,p) \leq \mu_*   \bigg]  = 1.
\end{equation*}
\end{lemma}

Postponing the proof of the lemmas, we show first that they imply Proposition~\ref{p.cell}.

\begin{proof}[{\bf Proof of Proposition~\ref{p.cell}}]
According to~\eqref{e.dvdcontp}, using also Lemma~\ref{l.detm} to control the constant in~\eqref{e.dvdcontp} on an event of full probability, it suffices to prove that 
\begin{equation} \label{e.cell2}
\P \left[ \forall p\in \Rd,\ \forall R>0, \ \limsup_{\delta \to 0} \sup_{B_{R/\delta}} \left| \delta v^\delta(\cdot,p) + \overline H(p) \right| =0 \right] = 1. 
\end{equation}
By~\cite[Lemma~5.1]{AS2}, to obtain~\eqref{e.cell2}, it suffices to show that 
\begin{equation} \label{e.cell0}
\P \left[ \forall p\in\Rd, \ \limsup_{\delta \to 0} \left| \delta v^\delta(0,p) + \overline H(p) \right| =0 \right] = 1. 
\end{equation}
Indeed, while the Hamiltonian in~\cite{AS2} is assumed to be convex in~$p$, the argument for~\cite[Lemma 5.1]{AS2} relies only on a $\P$-almost sure, uniform Lipschitz bound on $v^\delta(\cdot,p)$ (which we have in~\eqref{e.dvdLip}, using again Lemma~\ref{l.detm} to control the constant), and therefore the lemma holds in our situation notwithstanding the lack of convexity. 

\smallskip

To obtain~\eqref{e.cell0}, we consider the partition $\{ K_1,K_2,K_3,K_4\}$ of $\Rd$ given by~\eqref{e.Ki} and Proposition~\ref{p.partition} and check that, for each $i \in \{ 1,2,3,4\}$,
\begin{equation} \label{e.check.Ki}
\P \left[ \forall p\in K_i, \ \limsup_{\delta \to 0} \left| \delta v^\delta(0,p) + \overline H(p) \right| =0 \right] = 1. 
\end{equation}
In view of the formula~\eqref{e.Hbargoodform}, we see that:
\begin{itemize}

\item For~$i=1$, we consider two cases. If $\kappa \leq 0$, then $\mu_* = 0$ and, in view of the fact that $K_1 \subseteq B_1$, we obtain~\eqref{e.check.Ki} for $i=1$ from Lemmas~\ref{l.below-valley} and~\ref{l.above-hilltop}. If $\kappa > 0$, then~$\mu_*=\kappa>0$ and $\sigma_*=-1$ and we have $(\mu,-1) \in \A$ for all $0\leq \mu < \mu_*$, and thus~\eqref{e.check.Ki} for $i=1$ follows from~Lemmas~\ref{l.below.-1} and~\ref{l.above-hilltop}.

\item For $i=2$,~\eqref{e.check.Ki} is immediate from Lemmas~\ref{l.below.-1} and~\ref{l.above}.

\item For $i=3$, we get~\eqref{e.check.Ki} immediately from~Lemmas~\ref{l.below-valley} and~\ref{l.above}.

\item For $i=4$, the claim~\eqref{e.check.Ki} is a consequence of Lemmas~\ref{l.below.+1} and~\ref{l.above}.

\end{itemize}
This completes the argument. 
\end{proof}

We obtain each of the five auxiliary lemmas stated above by a comparison between the functions~$m_{\mu,\sigma}$ and~$v^\delta$, with the exception of Lemma~\ref{l.below-valley}, which is much simpler.

\begin{proof}[{\bf Proof of Lemma~\ref{l.below.+1}}]
Fix $(\mu,1) \in \mathcal A$ and $p\in \partial \overline m_{\mu,1}(\partial B_1)$. Select $e\in \partial B_1$ such that $p\in \partial \overline m_{\mu,1}(e)$. This implies that, for every $y\in \Rd$,
\begin{equation} \label{e.ptouchm}
\overline m_{\mu,1}(e) - p\cdot e = 0 \leq \overline m_{\mu,1}(y) - p\cdot y.
\end{equation}
Suppose that $V\in \Omega$ and $\delta >0$ are such that 
\begin{equation} \label{e.upass}
\theta:= \mu + \delta v^\delta (0,p) > 0.
\end{equation}
If $c>0$ is sufficiently small, then the function
\begin{equation*} \label{}
w(y) := v^\delta(y,p)-v^\delta(0,p) - c\theta \left(\left(1+|y|^2\right)^{1/2}-1\right)
\end{equation*}
satisfies
\begin{equation} \label{e.selc}
\left(\left| p+Dw \right|^2-1\right)^2-V(y) \leq -\delta v^\delta(y,p) + \frac14 \theta \quad \text{in} \ \Rd.
\end{equation}
Due to~\eqref{e.dvdbnd}, there exists $s>0$, independent of $\delta$, such that
\begin{equation} \label{e.Udeltout}
U := \left \{y\in \Rd\,:\, w(y) \geq -\frac{\theta}{4\delta} \right\} \subseteq B_{s/\delta}
\end{equation}
and specializing~\eqref{e.selc} to the domain $U$ yields, in view of the definition of $\theta$, that
\begin{equation}\label{e.wsub}
\left(\left| p+Dw \right|^2-1\right)^2-V(y) \leq \mu-\frac{\theta}{2} \quad \text{in} \ U.
\end{equation}
We observe next that, due to~\eqref{e.super} and~\eqref{e.Udeltout}, the function 
\begin{equation*} \label{}
\tilde m(y):= m_{\mu,1}\left(y, -se/\delta \right) - p\cdot y
\end{equation*}
is a supersolution of the equation
\begin{equation}\label{e.tildmsup}
\left(\left|p+D\tilde m\right|^2-1\right)^2-V(y) \geq \mu \quad \text{in} \ U.
\end{equation}
In view of $0\in U$,~\eqref{e.Udeltout},~\eqref{e.wsub} and~\eqref{e.tildmsup}, the comparison principle yields
\begin{equation*}
-\tilde m(0) = w(0)-\tilde m(0) \leq \max_{\partial U} \left ( w-\tilde m \right) = - \frac{\theta}{4\delta} + \max_{\partial U} \left (-\tilde m \right). 
\end{equation*}
Rearranging the previous inequality and using~\eqref{e.Udeltout}, we find that
\begin{equation*}
\sup_{y\in B_s} \left(p\cdot y + \delta m_{\mu,1}\left(0,\frac{-se}{\delta}\right)-\delta m_{\mu,1}\left(\frac{y}{\delta},\frac{-se}{\delta}\right)\right) \geq \frac{1}{4}\theta.
\end{equation*}
Notice that~\eqref{e.ptouchm} and the positive homogeneity of $\overline m_{\mu,1}$ implies that
\begin{equation*} \label{}
p\cdot y \leq \overline m_{\mu,1}(y+se) - \overline m_{\mu,1}(se).
\end{equation*}
Combining the previous two lines, we obtain
\begin{multline} \label{e.bgreps}
\sup_{y\in B_s} \left( - \overline m_{\mu,1}(se) + \delta m_{\mu,1}\left(0,\frac{-se}{\delta}\right)\right)\\ +\sup_{y\in B_s} \left(\overline m_{\mu,1}(y+se) -\delta m_{\mu,1}\left(\frac{y}{\delta},\frac{-se}{\delta}\right)\right) \geq \frac{1}{4}\theta.
\end{multline}
We have shown that~\eqref{e.upass} implies~\eqref{e.bgreps}. We therefore obtain the conclusion of the lemma by applying Proposition~\ref{p.shape}.
\end{proof}

\begin{proof}[{\bf Proof of Lemma~\ref{l.below.-1}}]
The proof is similar to Lemma~\ref{l.below.+1}. The difference is that we use~\eqref{e.super.global} rather than~\eqref{e.super}, which means that we do not have to take the vertex of $m_{\mu,\sigma}$ to be far away from the origin in the definition of the function $\tilde m$. The argument is therefore easier and the statement of the lemma is stronger. 

\smallskip

Fix $(\mu,-1) \in \mathcal A$ and $p\in \partial \overline m_{\mu,-1}(0)$. This implies that, for every $y\in \Rd$,
\begin{equation} \label{e.ptouchm.-1}
\overline m_{\mu,-1}(y) \geq p\cdot y.
\end{equation}
Suppose that $V\in \Omega$ and $\delta >0$ are such that 
\begin{equation} \label{e.upass.-1}
\theta:= \mu + \delta v^\delta (0,p) > 0.
\end{equation}
If $c>0$ is sufficiently small, then the function
\begin{equation*} \label{}
w(y) := v^\delta(y,p)-v^\delta(0,p) - c\theta \left(\left(1+|y|^2\right)^{1/2}-1\right)
\end{equation*}
satisfies
\begin{equation} \label{e.selc.-1}
\left(\left| p+Dw \right|^2-1\right)^2-V(y) \leq -\delta v^\delta(y,p) - \frac14 \theta \quad \text{in} \ \Rd.
\end{equation}
Due to~\eqref{e.dvdbnd}, there exists $s>0$, independent of $\delta$, such that
\begin{equation} \label{e.Udeltout.-1}
U := \left \{y\in \Rd\,:\, w(y) \geq -\frac{\theta}{4\delta} \right\} \subseteq B_{s/\delta}
\end{equation}
and restricting~\eqref{e.selc.-1} to the domain $U$ we obtain, in view of the definition of $\theta$, that
\begin{equation}\label{e.wsub.-1}
\left(\left| p+Dw \right|^2-1\right)^2-V(y) \leq \mu-\frac{\theta}{2} \quad \text{in} \ U.
\end{equation}
According to~\eqref{e.super}, the function 
\begin{equation*} \label{}
\tilde m(y):= m_{\mu,-1}\left(y, 0 \right) - p\cdot y
\end{equation*}
is a supersolution of the equation
\begin{equation}\label{e.tildmsup.-1}
\left(\left|p+D\tilde m\right|^2-1\right)^2-V(y) \geq \mu \quad \text{in} \ \Rd.
\end{equation}
In view of $0\in U$,~\eqref{e.Udeltout.-1},~\eqref{e.wsub.-1} and~\eqref{e.tildmsup.-1}, the comparison principle yields
\begin{equation*}
0= w(0)-\tilde m(0) \leq \max_{\partial U} \left ( w-\tilde m \right) = - \frac{\theta}{4\delta} + \max_{\partial U} \left (-\tilde m \right). 
\end{equation*}
Rearranging the previous inequality and using~\eqref{e.Udeltout.-1}, we find that
\begin{equation*}
\sup_{y\in B_s} \left(p\cdot y -\delta m_{\mu,-1}\left(\frac{y}{\delta},0\right)\right) \geq \frac{1}{4}\theta.
\end{equation*}
Using~\eqref{e.ptouchm.-1}, we get
\begin{equation} \label{e.bgreps.-1}
\sup_{y\in B_s} \left( \overline m_{\mu,-1}(y) -\delta m_{\mu,-1}\left(\frac{y}{\delta},0\right) \right) \geq \frac{1}{4}\theta.
\end{equation}
We have shown that~\eqref{e.upass.-1} implies~\eqref{e.bgreps.-1}. We therefore obtain the conclusion of the lemma from Proposition~\ref{p.shape}.
\end{proof}

\begin{proof}[{\bf Proof of Lemma~\ref{l.below-valley}}]
Fix $p\in \Rd$, $V\in \Omega$ for which $\inf_{\Rd} V = 0$ and $\sup_{\Rd} V = \overline v$ and let $\theta > 0$. Select $y_\theta\in \Rd$ and a number $r>0$ such that $\sup_{B_r(y_\theta)} V \leq \theta$. Let $\varphi$ be any smooth function on $B_r(y_\theta)$ such that $\varphi(x)\rightarrow +\infty$ as $x \to \partial B_r(y_\theta)$. Then $v^\delta(\cdot,p) - \varphi$ attains a local maximum at some point $y \in B_r(y_\theta)$. The equation~\eqref{e.appcell} then gives
\begin{equation}\label{e.belowthe}
\delta v^\delta(y,p)\leq \delta v^\delta(y,p) + \left(\left| p+D\varphi(y) \right|^2 -1 \right)^2 \leq V(y) \leq \theta.
\end{equation}
Letting $r\to 0$, we obtain that 
\begin{equation*} \label{}
\delta v^\delta (y_\theta,p) \leq \theta. 
\end{equation*}
In view of~\eqref{e.dvdLip}, we have 
\begin{equation*}
-\delta v^\delta(0,p) \geq -\theta -\delta \left|v^\delta(y_\theta,p)-v^\delta(0,p)\right|\geq  -\theta - C\delta |y_\theta|
\end{equation*}
where $C>0$ depends only on an upper bound for $|p|$ and $\overline v$. Sending first $\delta \to 0$ and then $\theta \to 0$ yields
\begin{equation} \label{e.belowzero}
\liminf_{\delta \to 0} -\delta v^\delta(0,p) \geq 0. 
\end{equation}
We have shown that $\inf_{\Rd} V = 0$ and $\sup_{\Rd} V =\overline v$ imply~\eqref{e.belowzero} for all $p\in \Rd$. We therefore obtain the conclusion of the lemma by an appeal to~Lemma~\ref{l.detm}. 
\end{proof}

\begin{proof}[{\bf Proof of Lemma~\ref{l.above}}]
The argument is similar to that of Lemma~\ref{l.below.+1}. We fix $(\mu,\sigma) \in \mathcal A$ and $p\in \partial \overline m_\mu(\partial B_1)$. Select $e\in \partial B_1$ such that $p \in \partial \overline m_{\mu,\sigma}(e)$. Since $\overline m_{\mu,\sigma}$ is positively homogeneous, this means that, for every $y\in \Rd$,
\begin{equation} \label{e.cripp}
\overline m_{\mu,\sigma}(e) - p\cdot e = 0 \leq \overline m_{\mu,\sigma}(y) - p\cdot y.
\end{equation}

\smallskip

We suppose that for fixed $V\in \Omega$ and $\delta>0$ we have
\begin{equation}\label{e.downass}
-\theta:=\mu+\delta v^\delta(0,p) <0.
\end{equation}
We define
$$
w(y):=v^\delta(y,p)-v^\delta(0,p)+c\theta\left( \left(1+|y|^2\right)^{1/2}-1\right),
$$
and notice that, for $c>0$ sufficiently small, $w$ satisfies
\begin{equation}\label{e.supc}
\left(|p+Dw|^2-1\right)^2-V(y) \geq -\delta v^\delta(y,p)-\frac14 \theta \quad \text{in}\ \Rd.
\end{equation}
By \eqref{e.dvdbnd}, there exists $s>0$, which independent of $\delta$, such that
\begin{equation}\label{e.Udoagain}
U:=\left\{y\in \Rd\,:\,w(y) \leq \frac{\theta}{4\delta}\right\} \subseteq B_{s/\delta}.
\end{equation}
In view of~\eqref{e.downass},~\eqref{e.supc} and~\eqref{e.Udoagain}, we have
\begin{equation}\label{e.wsup}
\left(|p+Dw|^2-1\right)^2-V(y) \geq \mu+\frac{\theta}{2} \quad \text{in}\ U.
\end{equation}
We next employ \eqref{e.sub}, \eqref{e.Udoagain}, and the fact that $\sigma^2 \leq 1$ to deduce that the function
$$
\tilde m(y):=-m_{\mu,\sigma}(y,se/\delta)-p\cdot y
$$
is a subsolution of the equation
\begin{equation}\label{e.mtildesub}
\left(|p+D\tilde m|^2-1\right)^2-V(y) \leq \mu \quad \text{in}\ U.
\end{equation}
The usual comparison hence implies
$$
\tilde m(0)=\tilde m(0)-w(0) \leq \max_{\partial U}(\tilde m - w)
$$
Rearranging the above and using \eqref{e.Udoagain} to achieve that
\begin{equation}\label{e.above.re}
\sup_{y\in B_s}\left(-p\cdot y - \delta m_{\mu,\sigma}\left(\frac{y}{\delta},\frac{se}{\delta}\right)+\delta m_{\mu,\sigma}\left(0,\frac{se}{\delta}\right)\right) \geq \frac14 \theta.
\end{equation}
By the symmetric property~\eqref{e.mmusym} of $m_{\mu,\sigma}$ and~\eqref{e.cripp}, we get
\begin{equation}\label{e.supcont}
\sup_{y\in B_s}\left(- \left(\delta m_{\mu,\sigma}\left(\frac{se}{\delta},\frac{y}{\delta}\right)-\overline m_{\mu,\sigma}(se-y)\right)
+\left( \delta m_{\mu,\sigma}\left(\frac{se}{\delta},0\right)-\overline m_{\mu,\sigma}(se)\right) \right) \geq \frac14 \theta.
\end{equation}
We have shown that~\eqref{e.downass} implies~\eqref{e.supcont}. The conclusion of the lemma therefore follows from~Proposition~\ref{p.shape}.
\end{proof}

\begin{proof}[{\bf Proof of Lemma~\ref{l.above-hilltop}}]

Fix $p\in \overline B_1$ and $V\in \Omega$ for which $\sup_{\Rd} V = \overline v$. Suppose $\theta,\delta > 0$ are such that 
\begin{equation}\label{e.downhillass}
\delta v^\delta(0,p) \leq -\mu_* - \theta.
\end{equation}
Define
\begin{equation*} \label{}
w(y):=v^\delta(y,p)-v^\delta(0,p)+c\theta\left( \left(1+|y|^2\right)^{1/2}-1\right),
\end{equation*}
and check that, if $c>0$ is sufficiently small, then $w$ satisfies
\begin{equation}\label{e.hillsupc}
\left(\left|p+Dw\right|^2-1\right)^2-V(y) \geq -\delta v^\delta(y,p)-\frac14 \theta \quad \text{in}\ \Rd.
\end{equation}
By \eqref{e.dvdbnd}, there exists $s>0$, which is independent of $\delta$, such that
\begin{equation}\label{e.Uinagain}
U:=\left\{y\in \Rd\,:\,w(y) \leq \frac{\theta}{4\delta}\right\} \subseteq B_{s/\delta}.
\end{equation}
In view of~\eqref{e.downhillass},~\eqref{e.hillsupc} and~\eqref{e.Uinagain}, we have
\begin{equation}\label{e.hillwsup}
\left(|p+Dw|^2-1\right)^2-V(y) \geq \mu_*+\frac{1}{2}\theta   \quad \text{in}\ U.
\end{equation}
Set $\eta:= \frac14 \min\{\theta,1\}$, select $y_\theta \in \Rd$ such that 
\begin{equation} \label{e.ytheta}
V(y_\theta) \geq \overline v - \eta = 1 - \kappa - \eta.
\end{equation}
According to~\eqref{e.mmuyz.-1} and~\eqref{e.mmuyz.+1}, $m_{\mu,0}(y,z) = |y-z|$ for every $\mu\geq 0$ and $y,z\in \Rd$. Define
\begin{equation*} \label{}
\tilde m(y):=-m_{0,0}(y,y_\theta)-p\cdot y=-|y-y_\theta|-p\cdot y.
\end{equation*}
We claim that  $\tilde m$ is a subsolution of the equation
\begin{equation}\label{e.hillmsub}
\left(|p+D\tilde m|^2-1\right)^2-V(y) \leq \mu_*+\frac{1}{4}\theta \quad \text{in}\ \Rd.
\end{equation}
In view of \eqref{e.sub}, it suffices to check~\eqref{e.hillmsub} at the vertex point $y_\theta$. We consider a smooth test function $\phi$ such that
\begin{equation*} \label{}
y\mapsto \tilde m(y) - \phi(y) = -|y-y_\theta|- \left( \phi(y) + p\cdot y  \right) \ \ \mbox{has a local maximum at} \ y = y_\theta.
\end{equation*}
It is evident that
\begin{equation*} \label{}
|p+D\phi(y_\theta)|^2  \leq 1.
\end{equation*}
Thus $\left( |p+D\phi(y_\theta)|^2  - 1\right)^2\leq 1$ and so we deduce from~\eqref{e.ytheta} that 
\begin{equation*} \label{}
\left( |p+D\phi(y_\theta)|^2  - 1\right)^2 - V(y_\theta) \leq 1 - V(y_\theta) \leq \kappa + \eta \leq \mu_* + \frac14\theta. 
\end{equation*}
This completes the proof of~\eqref{e.hillmsub}.

\smallskip

Applying the comparison principle, in view of $0\in U$,~ \eqref{e.hillwsup} and~\eqref{e.hillmsub}, we obtain
\begin{equation*} \label{}
\tilde m(0)=\tilde m(0)-w(0) \leq \max_{\partial U}(\tilde m - w)
\end{equation*}
Rearranging the above expression and using \eqref{e.Uinagain}, we deduce that
\begin{equation*}
\sup_{y\in B_s}\left(-p\cdot y - \delta \left|\frac{y}{\delta}-y_\theta\right|+\delta |y_\theta|\right) \geq \frac14 \theta.
\end{equation*}
By the usual triangle inequality, this yields
\begin{equation} \label{e.hillsupcont}
\sup_{y\in B_s} \left(p\cdot (-y) -|y|\right) \geq \frac14\theta  - 2\delta |y_\theta|.
\end{equation}
We have shown that~\eqref{e.downhillass} and $\sup_{\Rd} V = \overline v$ implies~\eqref{e.hillsupcont}. As $p\in \overline B_1$,~\eqref{e.hillsupcont} is impossible for $\delta < \theta/(8|y_\theta|)$.

\smallskip

We have shown that $\sup_{\Rd} V = \overline v$ implies that, for every $p\in\overline B_1$ and $\theta > 0$,
\begin{equation*} \label{}
\delta v^\delta(0,p) \geq -\mu_* - \theta \quad \mbox{for all} \quad 0<\delta < \theta/(8|y_\theta|).
\end{equation*}
Thus $\sup_{\Rd} V = \overline v$ implies
\begin{equation*} \label{}
\limsup_{\delta \to 0} \sup_{|p|\leq 1} - \delta v^\delta(0,p) \leq \mu_*.
\end{equation*}
We therefore obtain the statement of the lemma after an appeal to Lemma~\ref{l.detm}. 
\end{proof}

\noindent{\bf Acknowledgements.}
The third author was partially supported by NSF CAREER award \#1151919.

\bibliographystyle{plain}
\bibliography{nonconvex}

\end{document}